\documentclass[a4paper,12pt]{article}
\usepackage{amssymb}
\usepackage{mathrsfs}
\usepackage{amsmath}
\usepackage{amsfonts,amsthm,amssymb}
\usepackage{amsfonts}
\usepackage{graphics}
\usepackage{color}
\usepackage{graphicx}
\usepackage{epstopdf}
\usepackage{subfigure}

\newtheorem{lem}{Lemma}[section]
\newtheorem{thm}[lem]{Theorem}
\newtheorem{cor}[lem]{Corollary}

\newtheorem{Conjecture}[lem]{Conjecture}

\begin{document}

\title{Ramsey numbers for  complete graphs versus generalized fans}
\author{Maoqun Wang and Jianguo Qian\footnote{Corresponding author. E-mail: jgqian@xmu.edu.cn (J.G. Qian)}\\
\small School of Mathematical Sciences, Xiamen University, Xiamen 361005, PR China}
\date{}
\maketitle
{\small{\bf Abstract.}\quad For two graphs $G$ and $H$, let $r(G,H)$ and $r_*(G,H)$ denote the  Ramsey number and star-critical Ramsey number of $G$ versus $H$, respectively. In 1996, Li and Rousseau proved that $r(K_{m},F_{t,n})=tn(m-1)+1$ for $m\geq 3$ and sufficiently large $n$, where $F_{t,n}=K_{1}+nK_{t}$. Recently, Hao and Lin proved that $r(K_{3},F_{3,n})=6n+1$ for $n\geq 3$ and $r_{\ast}(K_{3},F_{3,n})=3n+3$ for  $n\geq 4$. In this paper, we show that  $r(K_{m}, sF_{t,n})=tn(m+s-2)+s$ for sufficiently large $n$ and, in particular, $r(K_{3}, sF_{t,n})=tn(s+1)+s$ for $t\in\{3,4\},n\geq t$ and $s\geq1$. We also show that $r_{\ast}(K_{3}, F_{4,n})=4n+4$ for $n\geq 4$ and establish an upper bound on $r(F_{2,m},F_{t,n})$.


\section{Introduction}

For two graphs $G$ and $H$, we denote by  $r(G,H)$ the {\it Ramsey number} of $G$ versus $H$, that is, the minimum integer $r$ such that any red/blue edge-coloring of $K_{r}$ contains a red $G$ or a blue $H$. By the definition, $K_{r-1}$ has a red/blue edge-coloring that contains neither a red copy of $G$ nor a blue copy of $H$. We call such a coloring a {\it $(G,H)$-free coloring}. For two vertex disjoint graphs $G_{1}$ and $G_{2}$, we denote by $G_{1}+G_{2}$ the join of $G_{1}$ and $G_{2}$, that is, the graph obtained from $G_{1}$ and $G_{2}$ by adding all edges between $V(G_{1})$ and $V(G_{2})$. Further, for a graph $G$ and positive integer $n$, the disjoint union of $n$ copies of $G$ is denoted by $nG$. For positive integer $n$, the {\it fan} $F_n$ is defined as $K_{1}+nK_{2}$. Generally, the {\it generalized fan}  $F_{t,n}$ is defined by $K_{1}+nK_{t}$. It is clear that $F_{2,n}=F_{n}$.

Ramsey theory is a fascinating branch in combinatorics. Most problems in this area are far from being solved, which stem from the classic problem of determining the number $r(K_n,K_n)$. In this paper we focus on the Ramsey numbers for  complete graphs versus generalized fans.

 In \cite{Y}, Li and Rousseau proved the following result.

\begin{thm}\label{thm 1.1}\cite{Y}
For any graphs $G$ and $H$, there is a positive integer $n(G,H)$ such that for any integer $n\geq n(G,H)$,
$$r(K_{2}+H,K_{1}+nG)=kn(\chi(H)+1)+1,$$
where $k$ is the number of the vertices in $G$ and $\chi(H)$ the chromatic number of $H$.
\end{thm}

\begin{Conjecture}\label{Conjecture 1.2}\cite{E}
For any $n\geq m\geq 3$, $r(K_{m},F_{n})=2n(m-1)+1$.
\end{Conjecture}

Choosing $G=K_{2}$ and $H=K_{m-2}$ in Theorem~\ref{thm 1.1}, we see that Theorem~\ref{thm 1.1} supports Conjecture~\ref{Conjecture 1.2} for any $m$ and sufficiently large $n$. In contrast to sufficiently large $n$, the conjecture is also confirmed for some small $m$'s as follows.

\begin{thm}\label{thm 1.3}
The following statement holds:
\begin{description}
  \item[(i)]  $r(K_{3},F_{n})=4n+1$ for any $n\geq 3$ \cite{L,Y};
  \item[(ii)]  $r(K_{4},F_{n})=6n+1$ for any $n\geq 4$ \cite{E};
  \item[(iii)]  $r(K_{5},F_{n})=8n+1$ for any $n\geq 5$ \cite{Chen};
  \item[(iv)]  $r(K_{6},F_{n})=10n+1$ for any $n\geq 6$ \cite{STY}.
\end{description}
\end{thm}

Let $K_{r-1}\sqcup S_{k}$ denote the graph obtained from $K_{r-1}$ by adding a new vertex $v$ and an edge  joining $v$ to each of $k$ vertices of $K_{r-1}$. For $r=r(G,H)$, the {\it star-critical Ramsey number} $r_{\ast}(G,H)$ is defined as the minimum integer $r_{\ast}$ such that any red/blue edge-coloring of $K_{r-1}\sqcup S_{r_{\ast}}$ contains a red $G$ or a blue $H$. This parameter was first introduced by Hook and Isaak \cite{Hook}, which could be viewed as a refinement of the Ramsey number. For results on star-critical Ramsey numbers involving $F_{t,n}$, we refer the readers to references \cite{Hao,Hao2,Hajhi,ZL,Yan}. In \cite{Hao}, Hao and Lin  determined the Ramsey number and the star-critical Ramsey number of $K_{3}$ versus $F_{3,n}$.

\begin{thm}\label{thm 1.4}\cite{Hao}
$r(K_{3}, F_{3,n})=6n+1$ for any integer $n\geq 3$ and $r_{\ast}(K_{3}, F_{3,n})=3n+3$ for any integer $n\geq 4$.
\end{thm}

Recently, Hamm et al. \cite{A} determined the Ramsey number of $sK_{m}$ versus $F_{t,n}$ for sufficiently large $n$ as follows.

\begin{thm}\label{thm 1.6}\cite{A}
Let $m$, $s$, $t$, $n$ be four positive integers with $m\geq 3$ and $tn\geq \max \{(m+1)(m+C(t,m)+t)+3, \frac{(m-1)^{2}}{m-2}((s-1)(m-1)+t)\}$, where $C(t,m)$ is a constant depending only on $t$ and $m$. Then $r(sK_{m}, F_{t,n})=tn(m-1)+s$.
\end{thm}

In 1975, Burr, Erd\H{o}s and Spencer \cite{SB} investigated Ramsey number for disjoint union of small graphs and showed that $r(nK_{3},nK_{3})=5n$ for $n\geq 2$. Li and Rousseau \cite{Y} first studied off-diagonal Ramsey number of fans. They showed that $4n+1\leq r(F_{m},F_{n})\leq 4n+4m-2$ for $n\geq m\geq 1$. Later in \cite{Q}, Lin and Li proved that $r(F_{2},F_{n})=4n+1$ for $n\geq 2$ and improved the upper bound of Li and Rousseau to be $r(F_{m},F_{n})\leq 4n+2m$ for $n\geq m\geq 2$. Recently, Chen, Yu and Zhao \cite{Yu} showed that $9n/2-5\leq r(F_{n},F_{n})\leq 11n/2+6$ for any $n$. For more results on Ramsey numbers involving $F_{t,n}$, we refer the readers to references \cite{D,SP,Z}.

In this paper, we  prove the following two results.
\begin{thm}\label{thm 1.5}
For any integer $n\geq4$, $r(K_{3}, F_{4,n})=8n+1$ and $r_{\ast}(K_{3}, F_{4,n})=4n+4$.
\end{thm}

\begin{thm}\label{thm 1.7}
Let $m$, $s$, $t$, $n$ be four positive integers with $n\geq m\geq 3$. Then $tn(m+s-2)+s\leq r(K_{m}, sF_{t,n})\leq (tn+1)(s-1)+r(K_{m},F_{t,n})$.
\end{thm}

As direct consequences of Theorem~\ref{thm 1.4}, Theorem~\ref{thm 1.5}-\ref{thm 1.7} and Theorem~\ref{thm 1.6}, Theorem~\ref{thm 1.7}, respectively, we have the following two corollaries.

\begin{cor}\label{cor 1.9}
$r(K_{3}, sF_{t,n})=tn(s+1)+s$ for any $t\in\{3,4\},n\geq t$ and $s\geq1$.
\end{cor}

\begin{cor}\label{cor 1.10}
Let $m$, $s$, $t$, $n$ be four positive integers with  $m\geq 3$ and $tn\geq \max \{(m+1)(m+C(t,m)+t)+3, \frac{(m-1)^{2}}{m-2}t\}$, where $C(t,m)$ is a constant depending only on $t$ and $m$. Then $r(K_{m}, sF_{t,n})=tn(m+s-2)+s$.
\end{cor}

Further, combining Theorem \ref{thm 1.7}  with Conjecture~\ref{Conjecture 1.2} and Theorem~\ref{thm 1.3}, we have the following corollary.

\begin{cor}\label{cor 1.8}
Let $m$, $s$, $t$, $n$ be four positive integers with $n\geq m\geq 3$. If $3\leq m\leq 6$, then $r(K_{m}, sF_{n})=2n(s+m-2)+s$. In general, if Conjecture~\ref{Conjecture 1.2} is true, then $r(K_{m}, sF_{n})= 2n(s+m-2)+s$.
\end{cor}

Finally, we prove the following result.

\begin{thm}\label{thm 1.11}
Let $t$, $n$, $m$ be three positive integers with $t\geq 3$ and $2m\geq \max \{(t+1)(t+C(t)+2)+3, \frac{(t-1)^{2}}{t-2}((n-1)(t-1)+2)\}$ where $C(t)$ is a constant depending only on $t$. Then $r(F_{m}, F_{t,n})\leq \max\{m,n\}+(t-1)(2m+n)+n+m$.
\end{thm}

In the following section, we give a proof of Theorem~\ref{thm 1.5} and, in the last section, we give the proofs of Theorem~\ref{thm 1.7} and Theorem~\ref{thm 1.11}.

\section{Proof of Theorem~\ref{thm 1.5} }

Let $G=(V,E)$ be a graph. For any subset $S\subseteq V$ or $S\subseteq E$, we use $G[S]$ to denote the subgraph of $G$ induced by $S$. A {\it neighbour} of a vertex $v$ is a vertex adjacent to $v$. The set of all neighbours of $v$ is denoted by $N_{G}(v)$ and the degree of $v$ is defined to be $d_{G}(v)=|N_{G}(v)|$. For a subset $U\subseteq V$, $N_{G}(v,U)$ denotes the set of all neighbors of $v$ in $U$ and $d_{G}(v,U)=|N_{G}(v,U)|$ denotes the degree of $v$ in $U$. In the definition of Ramsey number, a red/blue edge-coloring  of $K_{r}$ on vertex set $V$ can be associated with graphs $(V,R)$ and $(V,B)$, where $R$ and $B$ consist of all red and blue edges, respectively. Thus they are complementary to each other. Conversely, any graph $G$ and its complement $\overline{G}$ can be associated with a red/blue edge-coloring of a complete graph on $V$ in which the edges in $G$ and  $\overline{G}$ are colored in red and blue, respectively. Throughout the following, for a red/blue edge-coloring, we also use $R$ (or $B$) to denote the spanning subgraph induced by the red edges (or blue edges) if no confusion can occur.

The {\it chromatic surplus} $s(G)$ of a graph $G$ is the minimum size of a color class over all proper vertex-colorings of $G$ by using $\chi(G)$ colors. Burr \cite{S} proved that, for any connected graph $H$ of order $n\geq s(G)$,
\begin{equation}\label{1}
r(G,H)\geq (\chi(G)-1)(n-1)+s(G).
\end{equation}

A graph $H$ is called {\it $G$-good} if the equality in (1) holds.

Let $U_{1},U_{2},\ldots ,U_{\chi(G)}$ be the color classes of vertices under a proper vertex coloring of $G$ and $|U_{1}|=s(G)$. Denote
$$\tau(G)=\min \min_{v\in U_{1}} \min_{2\leq i \leq \chi(G)} d_{G}(v,U_{i}),$$
where the first minimum is taken over all proper vertex colorings of $G$ with $\chi(G)$ colors.

 %Under all proper vertex-coloring of $G$ with $|U_{1}|=s(G)$ and all the other color classes $U_{2},\ldots ,U_{\chi(G)}$, let $\tau(G)$ be the minimum degree of some vertex of $U_{1}$ in $U_{i}$ for $2\leq i \leq \chi(G)$. In order to prove the Theorem~\ref{thm 1.5}, we introduce the following lemmas.

We begin our proof with the following four lemmas.

\begin{lem}\label{lem 2.1}\cite{PJ}
For $n\geq m\geq 1$ and $n\geq 2$, $r(mK_{3}, nK_{4})=4n+2m+1$.
\end{lem}

\begin{lem}\label{lem 2.2}\cite{Hao,Hao2}
Let $G$ be a graph with $\chi(G)\geq 2$, and $H$ a connected graph of order $n\geq  s(G)$ with minimum degree $\delta(H)$. If $H$ is $G$-good, then $r_{\ast}(G,H)\geq (\chi(G)-2)(n-1)+\min \{n,\delta(H)+\tau(G)-1\}$.
\end{lem}

\begin{lem}\label{lem 2.3}
Let $G$ be a graph whose vertex set is partitioned into three subsets $V_{1}$, $V_{2}$ and $V_{3}$ with $G[V_{1}]=K_{4n-k}$ and $G[V_{2}]=K_{4n-k}$, where $1 \leq k\leq 3$ and $n\geq 4$. If $G$ has independence number $\alpha(G)$ at most two and contains no $F_{4,n}$, then for every vertex $w\in V_{3}$, $w$ is adjacent to all vertices of either $V_{1}$ or $V_{2}$.
\end{lem}
\begin{proof}
Suppose to the contrary that $w$ is not adjacent to $u$ of $V_{1}$ and $v$ of $V_{2}$ for some vertex $w$ of $V_{3}$. If $w$ has four nonadjacent vertices $u,u_{1},u_{2},u_{3}$ of $V_{1}$, then these four vertices must be adjacent to $v$, otherwise $\alpha(G)\geq 3$, which is a contradiction. In this way, $G[V_{2}\cup\{u,u_{1},u_{2},u_{3}\}]$ contains an $F_{4,n}$ with $v$ as its center, also a contradiction. Therefore, $w$ has at most three nonadjacent vertices in $V_{1}$. That is, $w$ has at least $4n-k-3$ neighbors in $V_{1}$. For the same reason, $w$ has at least $4n-k-3$ neighbors in $V_{2}$. We can find $n-2$ disjoint $K_{4}$ in $N_{G}(w,V_{1})$ and two disjoint $K_{4}$ in $N_{G}(w,V_{2})$, which together with $w$ form an $F_{4,n}$, a contradiction.
\end{proof}

\begin{lem}\label{lem 2.4}
For any $n\geq 4$, $B\supseteq 2K_{4n}$ for any $(K_{3},F_{4,n})$-free coloring of $K_{8n}$.
\end{lem}
\begin{proof}
Consider any ($K_{3},F_{4,n}$)-free coloring of $K_{8n}$. By Lemma~\ref{lem 2.1}, we know that $r(K_{3},nK_{4})=4n+3$, which implies that $d_{B}(u)\leq 4n+2$ or, equivalently,  $d_{R}(u)\geq 4n-3$  for any vertex $u\in V$. Since $R$ contains no red $K_{3}$, $N_{R}(u)$ contains a blue $K_{4n-3}$, denoted by $G_{1}$. Choose a vertex $v$ from $V(G_{1})$. Since $d_{R}(v)\geq 4n-3$, $N_{R}(v)$ also contains a blue $K_{4n-3}$, denoted by $G_{2}$. Therefore, we have $B\supseteq 2K_{4n-3}$.

 Let  $V_1=V(G_1), V_2=V(G_2)$ and $V_3=V(K_{8n})\setminus (V_{1}\cup V_{2})$. We note that $V_3$ consists of exactly six vertices. By Lemma~\ref{lem 2.3}, we know that every vertex in $V_3$ is blue-adjacent (adjacent by a blue edge) to all vertices of either $G_{1}$ or $G_{2}$. If at least four vertices in $V_3$ are blue-adjacent to all vertices of either $G_{1}$ or $G_{2}$, then either $B[V_3\cup V_1]$ or $B[V_3\cup V_2]$ contains a blue $F_{4,n}$, a contradiction. By symmetry, this means that exactly three vertices in $V_3$, say $x_{1},x_{2},x_{3}$, are blue-adjacent to all vertices in $V_1$ while the other three vertices, denoted by $x_{4},x_{5},x_{6}$, are blue-adjacent to all vertices in $V_2$.

If $x_{1}$ is blue-adjacent to at least four vertices in $V_2$, then $B$ has an $F_{4,n}$ with $x_{1}$ as its center, a contradiction. Hence, $x_{1}$ is blue-adjacent to at most three vertices in $V_2$. For the same reason, each of $x_{2}$ and $x_{3}$  is blue-adjacent to  at most three vertices in $V_2$. Equivalently, each of $x_1,x_{2}$ and $x_{3}$  is red-adjacent to  at least $4n-6$ vertices in $V_2$. Therefore, each pair of $x_{1}$, $x_{2}$ and $x_{3}$ are red-adjacent to a common vertex in $V_2$ as $|V_2|=4n-3$. This implies that $x_{1}x_{2}$, $x_{2}x_{3}$ and $x_{1}x_{3}$ are blue edges to avoid a red $K_{3}$. By symmetry, $x_{4}x_{5}$, $x_{5}x_{6}$ and $x_{4}x_{6}$ are also blue edges. As a result, $B[\{x_1,x_2,x_3\}\cup V_1]=K_{4n}$ and $B[\{x_4,x_5,x_6\}\cup V_1]=K_{4n}$, i.e., $B\supseteq 2K_{4n}$.
\end{proof}

\textbf{Proof of Theorem~\ref{thm 1.5}:} We first prove that $r(K_{3}, F_{4,n})=8n+1$ for $n\geq 4$. By (1),  $r(K_{3}, F_{4,n})\geq 8n+1$ follows immediately. We now need only to show the converse. Consider a red/blue edge-coloring of $K_{8n+1}$. Suppose to the contrary that $K_{8n+1}$ contains neither red $K_{3}$ nor blue $F_{4,n}$. Choose an arbitrary vertex $u\in V(K_{8n+1})$.  Since the red neighborhood (the neighbour adjacent by red edges) $N_{R}(u)$ of $u$ induces a blue complete graph, we have $d_{R}(u)\leq 4n$ and, hence, $d_{B}(u)\geq 4n$. Let $K_{8n+1}=K_{8n}+u$. It is clear that the red/blue edge-coloring restricted on $K_{8n}$ is also a ($K_{3},F_{4,n}$)-free coloring. So by lemma~\ref{lem 2.4}, we have $B-u\supseteq 2K_{4n}$. Therefore, $K_{8n+1}$ contains a blue $F_{4,n}$, a contradiction.

%Lemma~\ref{lem 2.1} tells that $r(K_{3},nK_{4})=4n+3$, which implies that $d_{B}(u)\leq 4n+2$ for any vertex $u\in V$. On the other hand, for any vertex $u\in V$, we have $d_{R}(u)\leq 4n$ since its red neighborhood $N_{R}(u)$ induces a blue complete graph. Thus, any vertex $u\in V$ satisfies that $4n \leq d_{B}(u)\leq 4n+2$ and $4n-2 \leq d_{R}(u)\leq 4n$. Since $R$ contains no red $K_{3}$, $N_{R}(u)$ induces a blue complete graph with at least $4n-2$ vertices. Take a subset $V_{1}\subseteq N_{R}(u)$ of size $4n-2$, and choose a vertex $v\in V_{1}$ which  has at least $4n-3$ neighbors in $N_{B}(u)$ since $d_{R}(v)\geq 4n-2$. It is easy to see that $B[N_{R}(u)]$ and $B[N_{R}(v)]$ induces two blue complete graphs of order at least $4n-2$. Take $U_{1}$ and $V_{1}$ be two vertex sunset with order $4n-2$ in $N_{R}(u)$ and $N_{R}(v)$, respectively. Let $\{x_{1},x_{2},x_{3},x_{4},x_{4=5}\}=V\backslash (U_{1}\cup V_{1})$. By Lemma~\ref{lem 2.2}, we know that every vertex of $\{x_{1},x_{2},x_{3},x_{4},x_{4=5}\}$ is blue-adjacent to either $U_{1}$ or $V_{1}$ completely. Therefore, there are at least three vertices in $\{x_{1},x_{2},x_{3},x_{4},x_{4=5}\}$ are blue-adjacent to $U_{1}$ or $V_{1}$ completely. In either case, there exists a blue $F_{4,n}$ in $\{x_{1},x_{2},x_{3},x_{4},x_{4=5}\}\cup U_{1}$ or $\{x_{1},x_{2},x_{3},x_{4},x_{4=5}\}\cup V_{1}$, a contradiction which completes the proof.

Next, we prove that $r_{\ast}(K_{3},F_{4,n})=4n+4$. Note that $r_{\ast}(K_{3},F_{4,n})\geq 4n+4$ follows directly from lemma~\ref{lem 2.2}. We now prove its converse.

Consider any ($K_{3},F_{4,n}$)-free coloring for $K_{8n}$. Lemma~\ref{lem 2.4} implies that the blue graph $B\supseteq 2K_{4n}$. Let $V_{1}$ and $V_{2}$ be the vertex sets of the two blue $K_{4n}$'s, respectively. Add a new vertex $v$ and an edge joining $v$ to each of  $4n+4$ vertices in $V_{1}\cup V_{2}$. It suffices to show that, no matter how to color these $4n+4$ edges, we can always find either a red $K_{3}$ or a blue $F_{4,n}$.

Note that for $n\geq 4$, there are at least four edges between $v$ and $V_{i}$ for each $i\in\{1,2\}$. Further, $v$ is adjacent to at least ten vertices of either $V_{1}$ or $V_{2}$, say $V_{1}$. Take $\{u_{1},u_{2},u_{3},u_{4}\}\subset V_{1}$ and $\{v_{1},v_{2},v_{3},v_{4}\}\subset V_{2}$. We see that $v$ is red-adjacent to some vertex of $\{u_{1},u_{2},u_{3},u_{4}\}$. Otherwise, $\{v\}\cup V_{1}$ will induce a blue $F_{4,n}$. Without loss of generality, assume $vu_{4}$ is red. Similarly, assume $vv_{4}$ is red. If $v$ has at least four blue neighbors in $V_{1}$, then we get a blue $F_{4,n}$, a contradiction. So we assume that $v$ has at most three blue neighbors in $V_{1}$. For the same reason, assume that $v_{4}$ has at most three blue neighbors in $V_{1}$. In this way, $v$ and $v_{4}$ have a common red neighbor in $V_{1}$. Therefore, we get a red $K_{3}$, again a contradiction. This completes the proof.

\section{Proof of Theorem~\ref{thm 1.7} and Theorem~\ref{thm 1.11} }

We first prove Theorem~\ref{thm 1.7}. Since the complete multipartite graph $K_{(tn+1)s-1,\underbrace{tn,tn,\ldots ,tn}_{m-2}}$ does not contain $K_{m}$ and its complement $K_{(tn+1)s-1}\cup \underbrace{K_{tn}\cup K_{tn},\ldots ,K_{tn}}_{m-2}$ does not contain $sF_{t,n}$, we have $r(K_{m}, sF_{t,n})\geq tn(s+m-2)+s$.

We now prove that $r(K_{m}, sF_{t,n})\leq (tn+1)(s-1)+r(K_{m},F_{t,n})$. Let $N=(tn+1)(s-1)+r(K_{m},F_{t,n})$. For a red/blue edge-coloring of $K_{N}$, if  $K_{N}$ has a red $K_{m}$, then we are done. Otherwise, we have a blue $F_{t,n}$ since $N>r(K_{m},F_{t,n})$. Removing the blue $F_{t,n}$ from $K_N$, we get a red/blue edge-coloring complete graph of order $(tn+1)(s-2)+r(K_{m},F_{t,n})$, which contains a blue  $F_{t,n}$ if $s-2\geq 0$. Repeating this procedure, we finally get a complete graph of order $r(K_{m},F_{t,n})-(tn+1)$, which is obtained from $K_N$ by removing $s$ blue $F_{t,n}$'s. This means $r(K_{m}, sF_{t,n})\leq (tn+1)(s-1)+r(K_{m},F_{t,n})$, which completes our proof of Theorem~\ref{thm 1.7}.

 %As the assertion for $s=1$ is trivial, we assume that $s\geq 2$ and the assertion holds for smaller $s$; hence $r(K_{m},(s-1)F_{t,n})\leq (tn+1)(s-2)+r(K_{m},F_{t,n})$. Therefore, if there is no red $K_{m}$ in some red/blue edge-coloring of $K_{(tn+1)(s-1)+r(K_{m},F_{t,n})}$ on vertex set $V$, there must be a blue $(s-1)F_{t,n}$. Let $U$ be the vertex set of the blue $(s-1)F_{t,n}$ and $W=V\backslash U$. Then $|W|= r(K_{m},F_{t,n})$. Thus $K_{(tn+1)(s-1)+r(K_{m},F_{t,n})}[W]$ contain a blue $F_{t,n}$. This completes the proof.

 Next, we prove Theorem~\ref{thm 1.11}. We begin with the following two lemmas.
\begin{lem}\label{lem 2.6}\cite{Hajnal}
Let $n,t$ be two positive integers with $t\geq 2$ and $G$ be a graph of order $tn$. If $\delta(G)\geq (1-\frac{1}{t})tn$, then $G$ has $n$ vertex-disjoint $K_{t}$'s.
\end{lem}
\begin{lem}\label{lem 2.7}
Let $s,n,t$ be three positive integers with $t\geq 2$. Then $r(sK_{2},F_{t,n})=\max\{s,n\}+(t-1)n+s$.
\end{lem}
\begin{proof}
We separate the proof into two cases.

{\bf Case 1.} $n\geq s$.

In this case, $\max\{s,n\}+(t-1)n+s=tn+s$. Since the graph $K_{tn, s-1}$ does not contain $sK_{2}$ and its complement $K_{tn}\cup K_{s-1}$ does not contain $F_{t,n}$, we have $r(sK_{2},F_{t,n})\geq tn+s$.

We now prove $r(sK_{2},F_{t,n})\leq tn+s$ by induction on $s$. The case $s=1$ is trivial. We assume that $s\geq 2$ and assertion holds for $s-1$, i.e., $r((s-1)K_{2},F_{t,n})\leq tn+s-1$. For any red/blue edge-coloring of $K_{tn+s}$, if $K_{tn+s}$ contains no blue $F_{t,n}$, then by the induction hypothesis, $K_{tn+s}$ must contain a red $(s-1)K_{2}$ as $tn+s>r((s-1)K_{2},F_{t,n})$.
%If we suppose that there is no red $sK_{2}$ either, then we shall have a contradiction.
Let $M=\{u_1v_1,u_2v_2,\ldots,u_{s-1}v_{s-1}\}$ be the set of the $s-1$ red $K_2$'s, $M^*=\{u_{1},v_{1},u_{2},v_{2},\ldots u_{s-1},v_{s-1}\}$ and $W=V(K_{tn+s})\backslash M^*$. Then $|W|=tn-s+2$. We see that $W$ induces a blue $K_{tn-s+2}$, since otherwise $W$ would contain a red edge, which together with $M^*$ form a red $sK_{2}$.

If $v_{i}$ is red-adjacent to one vertex in $W$ for some $i$, then $u_{i}$ must be blue-adjacent to all other vertices in $W$, since otherwise we would obtain two independent red edges between $\{u_{i},v_{i}\}$ and $W$, which together with the $s-2$ red $K_2$'s in $M\setminus\{u_iv_i\}$ form a red $sK_{2}$. By symmetry, this implies that, for any $i\in\{1,2,\cdots,s-1\}$, either $u_i$ or $v_i$ is blue-adjacent to all but at most one vertex in $W$. By relabeling if necessary, we may assume that each $u_{i}$ is blue-adjacent to all but at most one vertex in $W$. Let  $W'$ be the set of all red neighbours of $u_{i}$  in $W$ (maybe empty). Then $|W'|\leq s-1$ and $|W\backslash W'|\geq (tn-s+2)-(s-1)\geq 3$. We thus have a vertex $w\in W$ that is blue-adjacent to each $u_{i}$. Let $G_{1}=B[W\cup \{u_{1},u_{2},\ldots u_{s-1}\}\setminus \{w\}]$. It is easy to see that $\delta(G_{1})\geq tn-s\geq (1-\frac{1}{t})tn$. So by Lemma~\ref{lem 2.6}, we obtain $n$ vertex-disjoint $K_{t}$'s in $G_{1}$. Therefore, $G_{1}+w$ contains a blue $F_{t,n}$ with $w$ as its center, a contradiction.

{\bf Case 2.} $n < s$.

In this case, $\max\{s,n\}+(t-1)n+s=(t-1)n+2s$. The desired lower bound $r(sK_{2},F_{t,n})\geq (t-1)n+2s$ follows from the fact that the graph $\overline{K}_{(t-1)n}\cup K_{2s-1}$ does not contain $sK_{2}$ and its complement $K_{(t-1)n}+\overline{K}_{2s-1}$ does not contain $F_{t,n}$.

To prove $r(sK_{2},F_{t,n})\leq (t-1)n+2s$, we show  by induction on $s\geq n$ that if a graph $\overline{G}$ of order $(t-1)n+2s$ contains no $F_{t,n}$, then $G$ contains $s$  independent edges. The assertion for the case $s=n$ follows directly from Case 1. We now assume that $s> n$. By deleting a pair of nonadjacent vertices $u$ and $v$ from $\overline{G}$, we have a subgraph $\overline{H}$ of $\overline{G}$ on $(t-1)n+2(s-1)$ vertices with $s-1\geq n$, which does not contain $F_{t,n}$. So by the induction hypothesis, $H$ contains $(s-1)K_{2}$, which together with the edge $uv$ yields an $sK_{2}$ in $G$.
\end{proof}

%\begin{lem}\label{lem 2.8}
%For any integer $n$, $m$ and $t$ with $m\geq t\geq 3$, if Conjecture~\ref{Conjecture 1.2} is true, then $r(F_{m},nK_{t})\leq (n+2m-1)(t-1)+n$.
%\end{lem}
%\begin{proof}
%Since the graph $K_{tn-1,\underbrace{t-1,\ldots t-1}_{2m-1}}$ does not contain $F_{m}$ and its complement $K_{tn-1}\cup \underbrace{K_{t-1}\cup \ldots \cup K_{t-1}}_{2m-1} $ does not contain $nK_{t}$. Therefore, $r(F_{m},nK_{t})\geq (n+2m-1)(t-1)+n$.

%We shall prove the inequality $r(F_{m},nK_{t})\leq (n+2m-1)(t-1)+n$ by induction on $n$. The case $n=1$ is obtained by  the Conjecture~\ref{Conjecture 1.2}. We assume that $n\geq2$ and the assertion holds for smaller $n$; hence $r(F_{m},(n-1)K_{t})\leq (n+2m-2)(t-1)+n-1$. If there is no red $F_{m}$ in some red/blue edge-coloring of $K_{(n+2m-1)(t-1)+n}$ on vertex set $V$, there must be a blue $(n-1)K_{t}$. Let $U$ be the vertex set of the blue $(n-1)K_{t}$ and $W=V\backslash U$. Then $|W|=(n+2m-1)(t-1)+n-(n-1)t=2m(t-1)+1$. By Conjecture~\ref{Conjecture 1.2}, $K_{(n+2m-1)(t-1)+n}[W]$ contain a red $F_{m}$ or a blue $K_{t}$. Therefore, $r(F_{m},nK_{t})\leq (n+2m-1)(t-1)+n$. This completes the proof.
%\end{proof}

 Let $N=\max\{m,n\}+(t-1)(2m+n)+n+m$, and consider a red/blue edge-coloring of $K_{N}$. Suppose to the contrary that $R$ contains no $F_{m}$ and $B$ contains no blue $F_{t,n}$. By Lemma~\ref{lem 2.7}, we know that $d_{R}(u)\leq \max\{m,n\}+(t-1)n+m-1$ for any $u\in V(K_N)$ since the red neighbours of $u$ induce no $mK_2$. Furthermore, by Theorem~\ref{thm 1.6}, we have $d_{B}(u)\leq 2m(t-1)+n-1$ for any $u\in V(K_N)$. Then $N\leq \max\{m,n\}+(t-1)n+m-1 + 2m(t-1)+n-1 +1<\max\{m,n\}+(t-1)(2m+n)+n+m$, a contradiction. This completes the proof of Theorem~\ref{thm 1.11}.

\section{Acknowledgements}
This work was supported by the National Natural Science Foundation of China [Grant numbers, 11971406, 12171402].

\end{document}